\documentclass[a4paper]{article} 

\usepackage[utf8]{inputenc}
\usepackage[T1]{fontenc}
\usepackage{lmodern}

\usepackage{mymath,mytheorems,mytikz}
\usepackage[margin=1.5em,font=small]{caption}
\usepackage[vlined,ruled,titlenumbered,linesnumbered]{algorithm2e}
\usepackage{hyperref}

\tikzstyle{walk}=[semithick]
\tikzstyle{walka}=[very thick]
\tikzstyle{point}=[circle,fill,minimum size=.8mm,inner sep=0pt,outer sep=0pt]
\tikzstyle{pointa}=[circle,fill,minimum size=1mm,inner sep=0pt,outer sep=0pt]

\newcommand\walk[1]{
\if#1u\up\fi
\if#1d\down\fi
\if#1U\Up\fi
\if#1D\Down\fi
\if#1;\else\expandafter\walk\fi}

\DeclareRobustCommand\mdyck{$m$\nobreakdash-Dyck}
\DeclareRobustCommand\mluka{$m$\nobreakdash-Łuka\-sie\-wicz}
\newcommand\U{{\mathbf u}}
\newcommand\D{{\mathbf d}}
\newcommand\fold{\phi}
\newcommand\unfold{\phi^{-1}}

\newcommand\nr{{n'}}
\newcommand\hr{{h'}}

\title{A new bijection on \mdyck\ paths\\with application to random sampling}
\author{Axel Bacher}

\begin{document}

\maketitle

\begin{abstract}
We present a new bijection between variants of \mdyck\ paths (paths with steps
in $\{+1,-m\}$ starting and ending at height~$0$ and remaining at non-negative
height), which generalizes a classical bijection between Dyck prefixes and
pointed Łukasiewicz paths. As an application, we present a new random sampling
procedure for \mdyck\ paths with a linear time complexity and using a
quasi-optimal number of random bits. This outperforms Devroye's algorithm,
which uses $\bigoh(n\log n)$ random bits.
\end{abstract}

\section{Introduction}

Dyck paths---paths with steps in $\{\nearrow,\searrow\}$ which start and end at
height~$0$ and remain above the $x$-axis---are a cornerstone object in
combinatorics. They are counted by the ubiquitous Catalan numbers and are in
bijection with hundreds of other objects, among them binary plane trees (see
\cite{stanley2} for a list).

What makes Dyck paths especially interesting is their rich combinatorics. For
instance, the classical proof of the formula
\smash{$\frac1{n+1}\binom{2n}{n}$} for the Catalan numbers relies on the Cycle
Lemma, which operates on Dyck paths. Many bijections exist between variants of
Dyck paths (such as Dyck prefixes, which are not constrained to end at zero),
for instance relying on the \emph{Catalan decomposition} illustrated in
Figure~\ref{fig:catalan}. Many examples can be found, among others, in
\cite[Chapter~9]{lothaire}.

\def\up{ -- ++(1,1) node [point] {}}
\def\down{ -- ++(1,-1) node [point] {}}
\def\Up{ -- ++(1,1) node [pointa] {}}
\def\Down{ -- ++(1,-1) node [pointa] {}}

\begin{figure}[htb]
\hfil%
\begin{tikzpicture}[scale=.2]
\foreach \y in {0,...,5} \draw[help lines] (0,\y) -- ++(22,0);
\draw[walk] (0,0) node [point] {} \walk uuduuddduududuuudduuud;;
\draw[red,walka] (8,0) node [pointa] {} \walk U;
(13,1) node [pointa] {} \walk U;
(18,2) node [pointa] {} \walk UU;;
\end{tikzpicture}
\caption{A Dyck prefix of height $h = 4$ with its Catalan decomposition
$q_0{\color{red}\U}\dotsm{\color{red}\U}q_h$, where $q_0,\dotsc,q_h$ are
(possibly empty) Dyck paths.} \label{fig:catalan}
\end{figure}

A natural generalization of Dyck paths are paths with steps in $\{+1,-m\}$ for
some integer~$m\ge1$. These are called \emph{\mdyck\ paths} \cite{duchon} and
are in bijection with $m+1$-ary trees. Their counting sequence is called the
Fuss-Catalan numbers, equal to $\smash{\frac1{mn+1}\binom{(m+1)n}{n}}$.
The goal of this paper is to present a new bijection on \mdyck\ paths, which
corresponds in the case $m = 1$ to a well-known bijection between variants of
Dyck paths (Dyck prefixes and pointed Łukasiewicz paths), based on the Catalan
decomposition.

\bigskip

As an application of our bijection, we present a random sampling procedure for
\mdyck\ paths. Random sampling---finding an algorithm that outputs an element
of a given combinatorial class with a prescribed (usually uniform among the
objects of a given size) random distribution, as efficiently as possible---is
an important area of combinatorics with many theoretical and practical
applications (for instance, it can lead to conjectures on the properties of
large objects, or enable testing programs on random large inputs). The random
sampling of Dyck-like paths (or equivalently, plane trees) has attracted a lot
of attention. In the case of Dyck paths (or Motzkin paths, which allow
$\rightarrow$ steps), efficient techniques include anticipated rejection
\cite{barcucci,barcucci2} and Rémy's algorithm \cite{remy,motzkin2}. In a more
general case, including \mdyck\ paths, we have Devroye's algorithm
\cite{devroye} based on the Cycle Lemma.

The efficiency of a random sampling algorithm is measured in time, space and
random bits. The random bit complexity (as opposed to, for instance, the
number of calls to a uniform continuous random variable) is a realistic model
of the randomness consumed by the algorithm, developped for instance in
\cite{knuth}. By that measure, Devroye's algorithm uses $\bigoh(n\log n)$
random bits, since it involves drawing a uniform permutation. We show that our
algorithm has a linear cost with all three respects. In fact, we show that it
is \emph{asymptotically entropic} in the sense that the number of
random bits consumed is asymptotically equivalent to the entropy of the
\mdyck\ paths, which is an information-theoretical lower bound on the random
bit complexity.

\bigskip

This article is organized as follows. In Section~\ref{sec:defs}, we define the
classes of paths between which our bijections operates. The folding bijection
itself is presented in Section~\ref{sec:folding}. In Section~\ref{sec:random},
we show our random sampling algorithm and prove that it has linear complexity.
Finally, in Section~\ref{sec:limit}, we study the limit distribution of the
time complexity, which turns out to have unusual properties.

\section{Definitions} \label{sec:defs}

Throughout the paper, let $m\ge1$ be an integer. We consider paths with two
kinds of steps, $\U$ and $\D$, with respective heights $1$ and~$-m$. The
height of a path is defined as the sum of the heights of its steps. In the
rest of the section, we consider a path~$w$ with length~$n$ and height~$h$.
We introduce the Euclidean divisions of~$n$ and~$h$ by~$m+1$:
\begin{align*}
n &= (m+1) \nr + r\text;&
h &= (m+1) \hr + r\text.
\end{align*}
The remainders are the same since the heights of both steps~$\U$ and $\D$ are
congruent to~$1$ modulo~$m+1$. The height~$h$ is therefore determined by the
quotient~$\hr$, that we call the \emph{reduced height} of~$w$.

\bigskip

We say that the path~$w$ is an \emph{\mluka\ path} if every proper prefix~$p$
of~$w$ satisfies $h(p)\ge0$ but the whole path satisfies $h(w) < 0$. The final
height~$h$ ranges between~$-m$ and~$-1$, as determined by~$r$ (the reduced
height~$\hr$ is~$-1$). In particular, no \mluka\ path exists with a length
divisible by~$m+1$.

\bigskip

Finally, we say that the path~$w$ is an \emph{\mdyck\ prefix} if every
prefix~$p$ of~$w$ satisfies $h(p)\ge0$. If $r\ne0$, we call \emph{decoration}
of~$w$ a sequence $a_0,\dotsc,a_{\hr}$ of integers satisfying:
\[\begin{cases}
1 \le a_i\le m\text,&i = 0,\dotsc,\hr-1\text;\\
1 \le a_{\hr}\le r\text.
\end{cases}\]
We call the path~$w$ thus equipped a \emph{decorated path}. From the
constraints, we see that the number of possible decorations of~$w$ is
$rm^{\hr}$. In the case of Dyck prefixes ($m = 1$), there is only one possible
decoration for every Dyck prefix of odd length and zero for prefixes of even
length.

\section{The folding bijection} \label{sec:folding}

The goal of this section is to provide a bijection between the two objects
defined above, decorated \mdyck\ prefixes and pointed \mluka\ paths (i.e.,
with a distinguished step). We then provide an enumeration result as a first
application. In the case $m = 1$, our bijection reduces to a well-known
bijection between Dyck prefixes of odd length and pointed Łukasiewicz paths
\cite[Chapter~9]{lothaire}.

Let $w$ be an \mdyck\ prefix equipped with a decoration $(a_0,\dotsc,a_k)$.
Write~$w$ in the form:
\begin{equation} \label{w}
w = p\s\U q_0\dotsm\U q_k\text,
\end{equation}
where, for $i=0,\dotsc,k$, the path~$q_i$ is an \mdyck\ prefix of
height~$a_i-1$ (this is done by identifying first the factor~$\U q_k$ as the
smallest suffix of~$w$ of height~$a_k$ and working backwards to get the other
factors). Let $\fold(w)$ be the path:
\begin{equation} \label{wp}
\fold(w) = p\s q_0\D\dotsm q_k\D\text,
\end{equation}
pointed on the first step of~$q_0\D$. We call this operation \emph{folding}
the path~$w$.

To recover the path~$w$ from the pointed path~$\fold(w)$, let $pq$ be the
factorization of~$\fold(w)$ obtained by cutting before the pointed step. Let
$q_0\D$ be the smallest \mluka\ prefix of~$q$; repeat this process to get the
factorization \eqref{wp}. The path~$w$ is then recovered as \eqref{w} and the
decoration $a_0,\dotsc,a_k$ as the heights of the factors $\U q_i$. We call
this operation \emph{unfolding} the pointed path~$\fold(w)$.
The folding operation is illustrated in Figure~\ref{fig:fold}.

\def\down{ -- ++(1,-3) node [point] {}}
\def\Down{ -- ++(1,-3) node [pointa] {}}

\begin{figure}[htb]\small\hfil%
\begin{tikzpicture}[scale=.2,baseline=0cm]
\foreach \y in {0,...,12} \draw[help lines] (0,\y) -- ++(22,0);
\draw[walk] (0,0) node [point] {} \walk uuuduuuuuuuuduuuuuuudu;;
\draw[walka,red]
(8,4) node (a) [pointa] {} \walk U;
(13,5) node (b) [pointa] {} \walk U;
(16,8) node (c) [pointa] {} \walk U;
(22,10) node (d) [pointa] {};
\draw[very thin,red] (a) -- ++(16,0) (b) -- ++(11,0) (c) -- ++(8,0) (d) --
++(2,0);
\node at (24,4.5) [right,red] {$a_0=1$};
\node at (24,6.5) [right,red] {$a_1=3$};
\node at (24,9) [right,red] {$a_2=2$};
\node at (11,-2)
{$p\s {\color{red}\U}q_0{\color{red}\U}q_1{\color{red}\U}q_2$};
\end{tikzpicture}\hfil\hfil%
\begin{tikzpicture}[scale=.2,baseline=0cm]
\foreach \y in {0,...,8} \draw[help lines] (0,\y) -- ++(22,0);
\draw[walk] (0,0) node [point] {} \walk uuuduuuuuuudduuduuudud;;
\draw[walka,red] (8,4) node (a) [pointa] {}
(12,4) node [pointa] {} \walk D;
(15,3) node [pointa] {} \walk D;
(21,1) node [pointa] {} \walk D;
(a) circle (.55);
\node at (11,-2) {$p\s q_0{\color{red}\D}q_1{\color{red}\D}q_2{\color{red}\D}$};
\end{tikzpicture}\hfil
\caption{Left: a decorated $3$-Dyck prefix of length~$22$ and height~$10$ ($\nr
= 5$, $\hr = 2$ and $r = 2$). Right: its image by the folding operator~$\fold$,
a pointed 3-Łukasiewicz path.}
\label{fig:fold}
\end{figure}
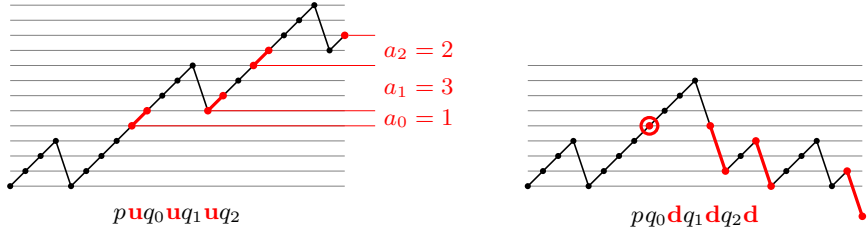

\begin{theorem}
The folding operation~$\fold$ is a bijection from decorated \mdyck\ prefixes
to pointed \mluka\ paths.
\end{theorem}

\begin{proof}
Let $w$ be an \mdyck\ prefix of height~$h$ written as \eqref{w} and, for
$i=0,\dotsc,k$, let~$a_i$ be the height of~$\U q_i$ (for now, without any
constraint on~$k$ or the $a_i$'s). The proof of the lemma hinges on the three
following facts.
\begin{itemize}
\item For $i = 0,\dotsc,k$, the path~$q_i\D$ is \mluka\ if and only if $a_i\le
m$.
\item Since $\fold(w)$ is obtained from~$w$ by turning $k+1$ up steps into
down steps, we have $\hr(\fold(w)) = -1$ if and only if $\hr(w) = k$.
\item If the factors $q_i\D$ are \mluka\ and if $\hr(\fold(w)) = -1$, the
folded path~$\fold(w)$ is \mluka\ if and only if the starting height of the
factor~$q_k\D$ is nonnegative. Since the final height is~$r - m - 1$, this is
equivalent to~$a_k\le r$.
\end{itemize}

Together, these facts show that the folding of a decorated \mdyck\ prefix is
an \mluka\ path and, conversely, that unfolding a pointed \mluka\ path yields
a decorated \mdyck\ prefix. Moreover, the first fact shows that the folding
and unfolding operations are inverse.
\end{proof}

\begin{proposition} \label{prop:enum}
The number $L_n$ of \mluka\ paths of length~$n$ is:
\begin{equation} \label{L}
L_n = \frac rn\binom{n}{\nr}\text.
\end{equation}
Let $P_n(u) = \sum_wu^{\hr(w)}$ where the sum runs over all \mdyck\ prefixes of
length~$n$. We have:
\begin{equation} \label{P}
P_n(m) = \binom{n}{\nr}\text.
\end{equation}
\end{proposition}

When $m = 1$, one recovers the well-known results on the enumeration of Dyck
paths (the evaluation $P_n(1)$ is simply the number of Dyck prefixes).

\begin{proof}
Let $q\D$ be an \mluka\ path of length~$n$. We transform it into the path~$\U
q$, which is a path of height~$r$ staying strictly above its origin. This fits
the conditions of the Cycle Lemma, which states that a proportion $r / n$ of
the paths with $\nr$ down steps have that property. This gives the formula
for~$L_n$.

To enumerate \mdyck\ prefixes, we use our bijection when $r\ne0$. There
are $nL_n$ pointed \mluka\ paths and therefore $nL_n$ decorated \mdyck\
prefixes. Every prefix with reduced height~$\hr$ has $rm^{\hr}$ possible
decorations, which gives the result.

If $r = 0$, this breaks down because there are no \mluka\ paths. %In this case,
The \mdyck\ prefixes of length~$n$ are then the prefixes of length~$n-1$ plus a
single $\D$ or $\U$ step. This gives $P_n(u) = (1 + u)P_{n-1}(u)$, which
implies the formula.
\end{proof}

\section{Random sampling} \label{sec:random}

In this section, we show how to use the unfolding bijection to build a very
efficient random sampling algorithm for \mdyck\ paths. In fact, our algorithm
returns a uniformly distributed \mluka\ path; to draw an \mdyck\ path of
length~$n$, it suffices to draw an \mluka\ path of length~$n+1$ and delete the
final~$\D$ step. In the case $m = 1$, the algorithm already appeared in
\cite{motzkin2}, but the time complexity analysis is new.

\begin{algorithm}[htb]
\DontPrintSemicolon
\caption{Random \mluka\ path} \label{algo}
\KwIn{A length~$n$ not divisible by~$m+1$}
\KwOut{A random \mluka\ path of length~$n$}
$w \leftarrow \epsilon$\;
\For{$i = 1,\dotsc,n$}{
    $w \leftarrow\bigl(w\U$ with probability $\frac m{m+1}$, $w\D$ with
probability $\frac1{m+1}\bigr)$\;
    \If{$h(w) < 0$}{
	draw uniformly a point in~$w$\;
	$w \leftarrow \unfold(w)$ (forget the decoration)\;
    }
}
draw uniformly a decoration of~$w$\;
$w \leftarrow \fold(w)$ (forget the point)\;
\Return{$w$}\;
\end{algorithm}

% It not \textit{a priori} obvious that this algorithm is correct, nor that it
% is efficient. 

\begin{theorem} \label{thm:correct}
Algorithm~\ref{algo} returns a uniformly distributed \mluka\ path of
length~$n$.
\end{theorem}

Before proving the theorem, we state our results on the complexity. We choose
two models of complexity, which account for the overwhelming majority of the
execution time in practice: the number~$R_n$ of random bits drawn and the
number~$M_n$ of memory accesses. We show that both complexities are linear and
derive their limit laws.

Let $\beta$ be the number of random bits necessary to draw a Bernoulli
variable of parameter \smash{$\frac1{m+1}$}. Moreover, let $S$ be an
inhomogeneous Poisson process on~$(0,1]$ with density~$\lambda(x)
=
\ifrac1{2x}$. Let $X$ be the random variable:
\[X = \sum_{x\in S}\unif[0,x]\text,\]
where all uniforms are independent from each other and from~$S$. Let
$U\sim\unif[0,1]$ independent from~$S$.

\begin{theorem} \label{thm:analysis}
The random variables $R_n$ and $M_n$ satisfy, as $n$ tends to infinity:
\begin{align*}
\frac{R_n}{n}&\tod\beta\text;&
\frac{M_n}{n}&\tod1 + X + U\text.
\end{align*}
\end{theorem}

The cost $\beta$ of drawing a Bernoulli variable is bounded from below by its
entropy:
\[\beta\ge\eta\text,\qquad
\eta=-\tfrac1{m+1}\log_2\bigl(\tfrac1{m+1}\bigr)-
\tfrac{m}{m+1}\log_2\bigl(\tfrac{m}{m+1}\bigr)\text.\]
With Knuth and Yao's algorithm \cite{knuth}, it is possible, with sufficient
grouping, to reach a value of~$\beta$ as close to $\eta$ as desired. Since an
\mdyck\ path of length~$n$ has entropy asymptotically $\eta n$ (as can be seen
from \eqref{L}), the number of random bits drawn by the algorithm can thus be
made to be asymptotically optimal.

The random variable~$X$ is studied in more detail in Section~\ref{sec:limit};
there, we show that $\mathbb E(X) = 1/4$ and $\mathbb V(X) = 1/12$. With the
values $\mathbb E(U) = 1/2$ and $\mathbb V(U) = 1/12$, this entails estimates
for the expectation and variance of $M_n$:
\begin{align*}
\mathbb E(M_n)&\sim\frac{7n}4\text;&\mathbb V(M_n)&\sim\frac{n^2}6\text.
\end{align*}

The crucial point in the proof of both theorems is a loop invariant given
in the lemma below. Consider a uniformly distributed decorated \mdyck\
prefix; let~$w$ be its underlying path. Since there are $rm^{\hr(w)}$
possible decorations of~$w$ and~$r$ depends only on~$n$, the path~$w$ is
distributed with a probability proportional to~$m^{\hr(w)}$. We denote
by~$\Pi_n$ that distribution on the \mdyck\ prefixes.

\begin{lemma} \label{lem}
Let $0\le i\le n$ and let $r = i\bmod m+1$. After $i$ iterations of the
$\mathbf{for}$ loop, the path~$w$ is distributed according to~$\Pi_i$.
Moreover, let $B_i$ be the event that the $\mathbf{if}$ branch is taken. The
events~$B_i$ are independent and satisfy:
\[\proba(B_i) = \frac{r}{mi + r}\text.\]
\end{lemma}

\begin{proof}
We work by induction on~$i$. First, consider the path~$w$ after adding a
random step. Since a $\U$ step is $m$ times as likely to be drawn than a $\D$
step, the probability of a given path~$w$ to appear is proportional
to~$m^{\hr(w)}$.

We now study the distribution of~$w$ after the \textbf{if} branch,
distinguishing whether or not it was taken.
\begin{itemize}
\item If the \textbf{if} branch is not taken, the path~$w$ is an \mdyck\
prefix, distributed according to~$\Pi_{i+1}$.
\item If the \textbf{if} branch is taken, the path~$w$ is a uniformly
distributed \mluka\ path since $\hr(w) = -1$. After pointing and unfolding, it
is therefore a uniformly distributed decorated \mdyck\ prefix, which means
that it is distributed like~$\Pi_{i+1}$ after forgetting the point.
\end{itemize}
This shows that $w$ is distributed like~$\Pi_{i+1}$. Moreover, the probability
that the branch is taken and not taken are proportional to $L_nm^{-1}$ and
$P_n(m)$, respectively (see Proposition~\ref{prop:enum}), which gives the
value of~$\proba(B_i)$. The independence comes from the fact that $w$ does not
depend on whether the branch is taken.
\end{proof}

\begin{proof}[Proof of Theorem~\ref{thm:correct}]
According to Lemma~\ref{lem}, after the execution of the \textbf{for} loop,
the path~$w$ is distributed like~$\Pi_n$. Drawing a random decoration
therefore yields a uniformly distributed decorated \mdyck\ prefix. After
folding, the result is thus a uniformly distributed \mluka\ path.
\end{proof}

\begin{proof}[Proof of Theorem~\ref{thm:analysis}]
Let us begin with the random bit cost. There are three places in the algorithm
which contribute to it and we analyse them separately.
\begin{itemize}
\item\emph{Drawing steps} (Line~3). This costs $n\beta$ random bits.
\item\emph{Randomly pointing the path} (Line~5). This costs $\bigoh(\log i)$ if
done at the $i$th iteration of the loop. According to Lemma~\ref{lem}, this
happens with probability $\bigoh(1/i)$. Summing for $i=1,\dotsc,n$, the
average total cost is $\bigoh(\log^2n)$.
\item\emph{Randomly decorating the path} (Line~7). This costs $\bigoh(h)$,
where~$h$ is the height of~$w$. To estimate this, we use
\cite[Theorem~6]{banderier}. Since the probability of a path~$w$ is
proportional to $m$ raised to the power of its number of up steps, that path
is distributed like a random meander with drift zero (see the reference for
details), which proves that the average height is $\bigoh(\sqrt n)$.
\end{itemize}
Overall, only the cost of Line~3 is significant. Let us move on to the cost in
memory accesses, which also occur in three places.
\begin{itemize}
\item\emph{Writing steps} (Line~3). This costs $n$ memory accesses.
\item\emph{Unfolding the path} (Line~6). Unfolding a pointed path~$pq$ of
length~$i$ only requires accessing the part~$q$. Since the point is uniformly
drawn, the cost is uniformly distributed in $\{1,\dotsc,i\}$. Moreover,
observe that the probability for this to occur given in Lemma~\ref{lem}, when
averaged over $m+1$ consecutive values of~$i$, is equivalent to~%\sfrac1{2i}. 
$\ifrac1{2i}$.

Let $S_n$ be the set of sizes~$i$ such that $B_i$ holds. Since the $B_i$'s are
independent, the set $S_n/n$ converges to the Poisson point process~$S$.
Therefore, the number of memory accesses divided by~$n$ tends in distribution
to~$X$.
\item\emph{Folding the path} (Line~8). Folding a path into a pointed
\mluka\ path~$pq$ only requires to access the part~$q$. Since that path is
uniformly distributed, the length of~$q$ is uniformly distributed
in~$\{1,\dotsc,n\}$.
\end{itemize}
Summing all three contributions (which are independent) yields the result.
\end{proof}

\section{Properties of the limit distribution} \label{sec:limit}

This last section consists in the study of the limit law~$X$ involved in
Theorem~\ref{thm:analysis}.

\begin{theorem}
The cumulants of the variable~$X$ are:
\[\kappa_n(X) = \frac1{2n(n+1)}\text.\]
\end{theorem}

In particular, we have $\mathbb E(X) = \kappa_1(X) = 1/4$ and $\mathbb V(X) =
\kappa_2(X) = 1/12$. Moreover, the theorem can be reformulated in terms of the
cumulant generating function of~$X$:
\begin{equation} \label{K}
K(z) = \int_0^z\frac{e^y-1-y}{2y^2}\dy\text.
\end{equation}

\begin{proof}
We compute the cumulant generating function of~$X$ from its definition,
knowing that the moment generating function of $\unif[0,x]$ is
$(e^{xz}-1)/(xz)$:
\[K(z) = \int_0^1\biggl(\frac{e^{xz}-1}{xz}-1\biggr)\frac{\dx}{2x}\text,\]
which is equivalent to \eqref{K} by a change of variables. The cumulants are
extracted by Taylor expansion around~$z = 0$.
\end{proof}

Our final results concern the distribution function $F(x) = \proba(X\le x)$
and tail distribution $\bar F(x) = \proba(X > x)$.

\begin{theorem}
The function~$F$ satisfies, for $x > 0$, 
the differential equation:
\begin{equation} \label{equadiff}
F(x) + F\p(x) + 2xF\pp(x) = F(x-1)\text.
\end{equation}
For $0\le x \le 1$, its value is:
\begin{equation} \label{smallx}
F(x) = \sqrt{\frac{2e^{1-\gamma}}{\pi}}\sin\sqrt{2x}\text,
\end{equation}
where $\gamma$ is Euler's constant. As $x$ tends to infinity, the tail
distribution satisfies:
\begin{equation} \label{largex}
\bar F(x) = x^{-x}(\log x)^{-2x}(e/2)^{x + o(x)}\text.
\end{equation}
\end{theorem}

Note that the equation \eqref{equadiff}, with the initial conditions
\eqref{smallx}, suffices to determine~$F$. Indeed, working inductively on
the intervals $[n,n+1]$, it can be seen as an inhomogeneous ordinary linear
differential equation with initial conditions given by differentiability
at~$n$.

Moreover, we can deduce from that equation the singularity profile of~$F$:
since $F$ is not differentiable at~$0$, $F''$ is not differentiable at~$1$ due
to the term~$F(x-1)$. In the same way, $F$ has a singularity at every integer
point~$n$, where it is exactly $2n$ times continuously differentiable.

Finally, since $F$ is twice differentiable for~$x > 0$, the distribution~$X$
admits a density function~$f = F'$, which shares similar properties.

\begin{figure}[htb]\small\hfil%
\begin{tikzpicture}[scale=1.6]
\draw[help lines] (0,0) grid (2,1);
\draw (0,0) node [anchor=north east] {$0$} -- (0,1) node [anchor=east] {$1$};
\draw[-latex] (0,0) -- (1,0) node [below] {$1$} -- (2,0) node [right] {$x$};
\draw[semithick] plot[smooth] file {F0.table} node [right] {$F_0(x)$};
\end{tikzpicture}\hfil\hfil%
\begin{tikzpicture}[scale=1.6]
\draw[help lines] (0,0) grid (2,1);
\draw (0,0) node [anchor=north east] {$0$} -- (0,1) node [anchor=east] {$1$};
\draw[-latex] (0,0) -- (1,0) node [below] {$1$} -- (2,0) node [right] {$x$};
\draw[semithick] plot[smooth] file {F.table} node [right] {$F(x)$};
\end{tikzpicture}\hfil
\caption{Left: a plot of the function $F_0(x) =
\sqrt{2e^{1-\gamma}/\pi}\sin\sqrt{2x}$. Right: a plot of the distribution
function~$F(x)$ computed from the differential equation \eqref{equadiff}. The
function~$F(x)$ is equal to $F_0(x)$ until $x = 1$ and then deviates from it.}
\end{figure}
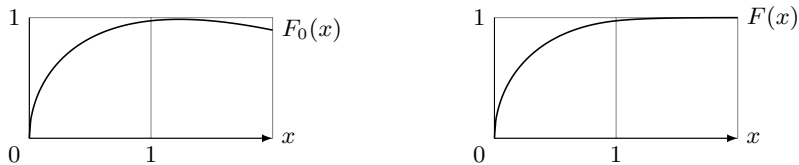

\begin{proof}
We prove these results using the Laplace transform of~$F$, which is given by
$\laplace F(z) = e^{K(-z)}/z$. From \eqref{K}, we get:
\[-2z^2{\laplace F}\p(z) = (e^{-z} - 3 + z)\laplace F(z)\text,\]
which translates into \eqref{equadiff} whenever $F$ is twice differentiable.
We now put the Laplace transform into the form\footnote{%
This comes from the fact that {$\int_0^1(e^{-y}-1+y)/y^2\dy +
\int_1^\infty e^{-y}/y^2\dy = \gamma$}, which can itself be derived by
integrating by parts twice to get {$-\int_0^\infty e^{-y}\log y\dy$}.
That integral is also linked to the exponential integral function, in which
the constant $\gamma$ famously plays a role (see for instance
\cite[Chapter~5]{abramowitz}).}
, valid for $z\not\in\mathbb R^-$:
\begin{align*}
\laplace F(z) &=
\frac{\sqrt{e^{1-\gamma}}}{z^{3/2}}\expraised{-\frac1{2z}}\exp\biggl(\int_z^\infty\frac{e^{-y}}{2y^2}\dy\biggr)\\
&= \frac{\sqrt{e^{1-\gamma}}}{z^{3/2}}\expraised{-\frac1{2z}} +
\bigoh\biggl(\frac{e^{-z}}{z^{7/2}}\biggr)
\end{align*}
as $\abs z$ tends to infinity in any direction. The main term transforms back
into~\eqref{smallx} (see \cite[29.3.78]{abramowitz}). The inverse transform of
the error term is supported for $x\ge1$ (since it is $\bigoh(e^{-z})$ as $z$
tends to infinity) and is twice differentiable (since, multiplied by~$z^2$, it
is integrable on~$i\mathbb R$). This proves \eqref{equadiff} and
\eqref{smallx}.

\bigskip

Getting the asymptotics for large~$x$ is trickier. We use a saddle point
approximation, widely used in statistics to estimate tail densities and
distributions \cite{daniels,lugannani}. General saddle point asymptotics are
described in detail in \cite[Chapter~VIII]{flajolet}. We compute the tail
distribution using the formula, valid for~$c > 0$:
\[\bar F(x) = \frac1{2\pi i}
\int_{c-i\infty}^{c+i\infty}\frac{e^{-xz + K(z)}}{z}\dz\text.\]
Write $\xi(z) = K(z) - \log z$. We choose $c$ to be the real point where the
integrand is smallest (the saddle point), given by:
\[\xi\p(c) = \frac{e^c - 1 - 3c}{2c^2} = x\text.\]
This entails:
\[c = \log x + 2\log\log x + \log 2 + o(1)\text.\]
Moreover, $\xi(c)$ and all its derivatives are asymptotic to~%
%$e^c/(2c^2)\sim x$.
$\ifrac{e^c}{2c^2}\sim x$.

Let $d = x^{-a}$ with $1/3 < a < 1/2$ and let $\bar F_0(x)$ and $\bar F_1(x)$
be:
\begin{align*}
\bar F_0(x) &= \frac1{2\pi i}\int_{c-id}^{c+id}e^{-xz + \xi(z)}\dz\text;\\
\bar F_1(x) &= \bar F(x) - \bar F_0(x)\text.
\end{align*}
We prove below that all the weight of the integral is concentrated
in~$\bar F_0(x)$ and that we have the saddle point approximation:
\begin{equation} \label{saddle}
\bar F(x) \sim \bar F_0(x) \sim \frac{e^{-xc +
\xi(c)}}{\sqrt{2\pi x}}\text.
\end{equation}
This evaluates to \eqref{largex} (the denominator is subsumed into the error
term).

To show that the approximation \eqref{saddle} is valid, we check that the
conditions detailed in \cite[Theorem~VIII.3]{flajolet} are satisfied.
\begin{itemize}
\item First, we need to check that $\bar F_0(x)$ satisfies \eqref{saddle}. To
do that, we do a Taylor expansion of $\xi(z)$ around the point $c$:
\[\xi(z) = \xi(c) + x(z-c) + \frac{\xi\pp(c)}{2}(z-c)^2 +
\bigoh\bigl(\xi'''\mspace{-1mu}(c)(z-c)^3\bigr)\text.\]
Since $\xi\ppp(c)\s d^3\sim xd^3\to0$, the error term tends in fact to zero,
uniformly for all $z$ such that $\abs{z-c} \le d$. This entails that $\bar
F_0(x)$ is approximated by the integral:
\[\bar F_0(x)\sim\frac{e^{-xc + \xi(c)}}{2\pi}
\int_{-d}^{d}\expraised{-\frac{\xi\pp\ns(c)\s t^2}{2}}\dt\text.\]
Since $\xi\pp(c)\s d^2\sim xd^2\to\infty$, the integral can be completed to
$\mathbb R$, which gives a Gaussian integral evaluating to~\eqref{saddle}.

\item Second, we need to check that the integral $\bar F_1(x)$ is negligible.
Using the estimate \eqref{saddle} of $\bar F_0(x)$, we compute:
\begin{align*}
\bgabs{\frac{\bar F_1(x)}{\bar F_0(x)}} &\le \sqrt{\frac{x}{2\pi}}
\int_{\abs t > d}\babs{e^{\xi(c+it) - \xi(c)}}\dt\\
&\le \sqrt{\frac{x}{2\pi}}\s e^{\rho(x)}
\int_{\abs t > d}\Babs{\frac c{c+it}}^{3/2}\dt\\
&= \bigoh\bigl(c\sqrt x\s e^{\rho(x)}\bigr)\text,
\end{align*}
where:
\[\rho(x) = \sup_{\abs t > d}\,
\real\biggl(\int_c^{c+it}\frac{e^y-1}{2y^2}\dy\biggr)\text.\]
This means that the ratio $\bar F_1(x)/\bar F_0(x)$ tends to zero as soon as $\rho(x)$ tends to
$-\infty$ sufficiently fast (like a power of~$x$). To show this, we set the
contour of integration to $c\to1\to1+it\to c+it$, which follows the direction
of steepest descent around the endpoints and avoids the singularity at zero.
The contribution of the interval $[1,1+it]$ is bounded, as is the contribution
of the term~$1/y^2$.  Grouping the other two intervals together, we find:
\begin{align*}
\rho(x)
&\sim\sup_{\abs t > d}\,
\int_1^c\real\biggl(\frac{e^{s+it}}{2(s+it)^2}-\frac{e^s}{2s^2}\biggr)\ds\\
&\le\sup_{\abs t > d}\,
\int_1^c\biggl(\frac{e^s}{2(s^2+t^2)} - \frac{e^s}{2s^2}\biggr)\ds\\
&\sim \sup_{\abs t > d}\,
-\frac{e^ct^2}{2c^2(c^2+t^2)}
= -\frac{e^cd^2}{2c^2(c^2+d^2)} \sim -\frac{x^{1-2a}}{\log^2 x}\text.
\end{align*}
This concludes the proof.\qedhere
\end{itemize}
\end{proof}

\bibliographystyle{abbrv}
\bibliography{biblio}{}

\end{document}